\documentclass[preprint, 12pt]{elsarticle}
\usepackage{amsmath,amsfonts,amsthm,graphicx}
\usepackage{setspace,xcolor}
\usepackage[a4paper, margin=1in]{geometry}
\usepackage[colorlinks=true,
linkcolor=blue,
citecolor=blue,
urlcolor=blue]{hyperref}
\newtheorem{theorem}{Theorem}[section]
\theoremstyle{definition}
\newtheorem{exam}{Example}[section]
\newtheorem{lemma}{Lemma}[section]
\newtheorem{corollary}{Corollary}[section]
\theoremstyle{definition}
\newtheorem{defn}{Definition}[section]
\newtheorem{Rem}{Remark}[section]
\newtheorem*{prf}{Proof}
\DeclareMathOperator{\tr}{\mathbf{trace}}
\begin{document}
\begin{abstract}		 
Identifying the collection of scalars that represent a non-negative matrix's eigenvalues is known as the non-negative inverse eigenvalue problem (NIEP). Conditions for the existence of a non-negative matrix with a certain spectrum are examined in this work. The classical NIEP restricted to non-negative matrices having a persymmetric structure is the persymmetric non-negative inverse eigenvalue problem (PNIEP). We resolve the open problem stated in \cite{Ana}, and furthermore, equivalence of the PNIEP and NIEP is established for trace-zero spectra of five complex numbers. Also we obtain new sufficient conditions for the realizability of certain classes of spectra with non-negative persymmetric matrix realization. Based on generic structural features of persymmetric matrices and their characteristic polynomials, our method is constructive in nature.  The effects of perturbations on imaginary part of complex eigenvalues are also analyzed and some perturbation results are derived.
\end{abstract}
\begin{frontmatter}
\title{\textsc{Non-negative persymmetric realizability of certain classes of spectra }}
\author[inst1]{Nayanthara }
\ead{nayan20@cusat.ac.in}
\author[inst1]{Noufal Asharaf}
\ead{noufal@cusat.ac.in}
\address[inst1]{Department of Mathematics, Cochin University of Science and Technology, Kerala, India.}
\end{frontmatter}
\section{Introduction}
Every square matrix of order $n$ has a set of $n$ eigenvalues, which are given by the roots of its characteristic polynomial. A natural converse question is: for a prescribed set of $n$ scalars (real or complex), does there exist a matrix of order $n$ having these scalars as its eigenvalues? This question has given rise to the Inverse Eigenvalue Problems (IEP), a central and active area of research within matrix theory. The IEP is concerned with identifying those sets of scalars that can occur as the eigenvalues of some matrix, along with determining the necessary and sufficient conditions for such occurrences. A set of scalars is said to be realizable if there exists a matrix whose eigenvalues coincide with the given set; such a matrix is called a realizing matrix.
The study of the IEP dates back nearly a century, with its origins often attributed to Kolmogorov’s work on stochastic matrices \cite{Kol}. Over time, the problem has gained significance due to its diverse applications. The problem becomes particularly intriguing when the matrix is required to satisfy additional constraints, such as non-negativity, symmetry, or persymmetry. The primary challenges lie in establishing solvability conditions and developing computational techniques for reconstructing matrices from prescribed spectral data, a task that remains both theoretically rich and computationally demanding.  The \textit{Non-negative inverse eigenvalue problem(NIEP)} is the problem of characterizing the set of scalars to be the eigenvalues of a non-negative matrix. The problem addresses the necessary and sufficient conditions for the existence of a non-negative matrix corresponding to a given set of scalars. The literature on the Non-negative Inverse Eigenvalue Problem (NIEP) is extensive, containing numerous sufficient conditions and a few necessary ones that characterize realizable spectra. Several studies have also explored the interrelations among these sufficient conditions. The NIEP has been solved for sets of small order $n.$ Nevertheless, a complete solution for general $n$ remains unknown.

The problem was resolved by Lowey and London in \cite{low} for the case $n=3$. Later, Meehan solved for $n=4$ in \cite{mee} and, independently, by Torre-Mayo et al. \cite{tor} in terms of the coefficients of the characteristic polynomial. The case of a list of five complex numbers with zero sum was solved by Laffey and Meehan in \cite{lame}. Laffey and Smigoc proved the general case for a specific set of complex numbers in \cite{lasmi}. This problem remains open for the case $n \geq 5.$ Imposing structural constraints on a matrix often increases the complexity of the inverse eigenvalue problem. In recent years, researchers have begun exploring the Persymmetric Non-negative Inverse Eigenvalue Problem (PNIEP), which restricts the classical Non-negative Inverse Eigenvalue Problem (NIEP) to matrices with a persymmetric structure. Persymmetric realizability imposes stricter conditions than general realizability. The study of persymmetric matrices offers valuable insights into how structural constraints affect the eigenvalues of a matrix and their interrelationships. From an applied perspective, persymmetric matrices naturally arise in fields such as signal processing, statistics, and combinatorial matrix theory, where symmetry patterns are inherent. As the general NIEP becomes increasingly complex and remains unsolved for matrices of higher orders, it is meaningful to investigate the problem within the restricted setting of persymmetric matrices, where the imposed structure may yield new theoretical insights and tractable subclasses. This work focuses on constructing non-negative matrices that possess a persymmetric structure corresponding to a given set of complex numbers.  A. I. Julio et al. in \cite{Ana} presented a notable result establishing that the NIEP and the PNIEP are equivalent for matrices of order $n=3$. They also provided a sufficient condition for a realizable list of five complex numbers with zero sum to serve as the spectrum of a persymmetric non-negative matrix. The authors raised the question of whether a set of five scalars with zero sum is persymmetrically realizable, assuming that it is already realizable but does not satisfy their sufficient condition. This open problem is addressed here, and a complete solution is given. Moreover, an equivalence between the NIEP and the PNIEP for trace-zero realizable spectra of order five is established. In addition, sufficient conditions are provided for the existence of a non-negative persymmetric matrix corresponding to certain sets of complex numbers. The realizability of the spectrum is also investigated when the imaginary part of a complex eigenvalue is perturbed. Perturbation results are established for both sets of four complex numbers with zero sum and sets of five complex numbers consisting of one element with a positive real part, the Perron eigenvalue, and all remaining elements with non-positive real parts.
\section{Non-negative Inverse Eigenvalue Problem}\label{sec:NIEP}
A square matrix $A:=(a_{ij})$ of order $n$ is said to be non-negative if $a_{ij} \geq 0$ for all $i,j=1,2,\ldots n.$ The $\tr(A)$ is the sum of it's diagonal entries, which is also equal to the the sum of it's eigenvalues.
For $\Delta :=\{\delta_1,\delta_2,\ldots,\delta_n\}$ the spectrum of A, we define the $k^{th}$ moment of $\Delta$ given by 
\[s_k (\Delta) :=  \delta_1^k+\delta_2^k+\cdots+\delta_n^k,\]
the sum of $k^{th}$ powers of elements in the spectrum, that is equal to the $\tr(A^k)$. If there exist a non-negative matrix $A$ whose spectrum is the set $\Delta$, then it is said to be \textit{realizable}. The matrix $A$ is called  the \textit{realizing matrix} of $\Delta$.
\begin{defn}
A \textit{persymmetric} matrix is an $n \times n$ matrix $A:=(a_{ij})$ whose entries satisfy
\[a_{ij}=a_{n-j+1, n-i+1} \quad \forall i,j \in {1,2,\ldots,n}\]
i.e., matrix A is \textit{persymmetric} if and only if $JAJ=A^{T},$ where $J$ is the $n \times n$ permutation matrix containing the units in the skew-diagonal and zeroes elsewhere. Persymmetric matrices are symmetric with respect to the anti-diagonal.
\end{defn}
\begin{defn}
Consider the monic polynomial $p(z) := z^n+a_{n-1}z^{n-1}+\cdots+a_1z +a_0$ of degree $n$ whose zeros are in $\Delta$ including multiplicities.  Newton’s identities provide a systematic relationship between the coefficients of this polynomial and the corresponding power sums of its roots, which can be expressed as a recurrence relation as follows:
\begin{equation}\label{newton}
	ka_{n-k}+s_1a_{n-k+1}+s_2a_{n-k+2}+\cdots+s_k a_n=0, \quad k=1,2,\ldots,n.
\end{equation}
\end{defn}

The literature on NIEP consists of many necessary and sufficient conditions which realizes a given set of scalars. Some of the established necessary conditions for a spectrum to be realizable are:
\begin{enumerate}[(i)]
\item The set will be closed under complex conjugation, $\Delta(A) = \Bar{\Delta}(A).$
\item Trace of powers of matrix $A$ will be non-negative i.e, $s_k(\Delta) = \sum \limits_{i=1}^{n} \delta_{i}^k \geq 0$ for all psitive integers $k.$
\item $\displaystyle \max_{1 \leq i\leq n}\{|\delta_i| :\delta_i \in \Delta\}$ belongs to $\Delta$, by Perron-Frobenius theorem \cite{Horn-John}.
\item An inequality involving $k^{th}$ moment proved by Lowey and London in \cite{low} and Johnson independently in \cite{john} given by
\[s_{k}^m \leq n^{m-1}s_{km}\] 
 for all positive integres $k$ and $m$.
\end{enumerate}
In addition to the necessary conditions, several sufficient conditions have also been established. The reader can refer \cite{niep} for details. The first known sufficient condition, due to Suleĭmanova \cite{Sulie} , asserts that any real list with exactly one positive value, all other values negative, and a non-negative sum is realizable as the spectrum of a nonnegative matrix, with the unique positive value serving as the Perron root. Other sufficient conditions for the Real-NIEP utilize partitions of the spectrum, where the list is divided into smaller sublists, each satisfying a Suleĭmanova-type condition, and these sublists are then combined to realize the full spectrum. Additional sufficient conditions, beyond partition methods, involve prescribing the diagonal entries of a matrix that realizes part of the spectrum and subsequently extending this partial realization to obtain the full spectrum. Analogous sufficient conditions have also been developed for the symmetric NIEP. The inclusion relationships and mutual independence among these well-known sufficient conditions have been investigated, revealing that while some conditions imply others, many remain incomparable.

The Nonnegative Inverse Eigenvalue Problem (NIEP) has been solved for spectra of small cardinality. By applying the Cauchy–Schwarz inequality, Loewy and London established a new necessary condition, thereby resolving the NIEP for sets of scalars of order $n = 3$.
\begin{theorem}\cite{low}
Let $\Delta:=\{\delta_1, \delta_2, \delta_3\}$be a list of complex numbers with $\delta_1 \geq |\delta_j|, j=2,3.$ Then $\Delta$ is realizable if and only if
$\Delta(A) = \Bar {\Delta}(A), 
s_1 \geq0\,\,  \text{and}\,\,  s_1^2 \leq 3s_2.$
\end{theorem}
R.Reams worked on the set of four complex numbers with zero sum and derived a characterization for their realizability.
\begin{theorem}\cite{reams}\label{realizing four}
Let $\Delta:=\{\delta_1, \delta_2, \delta_3.\delta_4\}$ be a list of complex numbers. Then $\Delta$ with $s_1=0$ is realizable if and only if
\[s_2\geq 0, \quad s_3\geq 0, \quad 4s_4\geq s_2^2\]
\end{theorem}
E.Meehan completely solved NIEP for $n=4$ in \cite{mee}. Later Laffey and Meehan presented a pioneering work by analyzing spectra consisting of five complex scalars satisfying $s_1 = 0.$  
\begin{theorem}\label{realizing three}
\cite{lame}
Let $\Delta:=\{\delta_1, \delta_2, \delta_3,\delta_4,\delta_5\}$ be a list of five complex numbers and let $s_k = \delta_1^k+\delta_2^k+\delta_3^k+\delta_4^k+\delta_5^k$ for $k=1,2,\ldots.$ Assume $s_1 =0$. Then $\Delta$ is the spectrum of a non-negative $5 \times 5$ matrix if and only if the following three conditions are satisfied:
\begin{enumerate}[(i)]
\item $s_k \geq 0,\quad k =2,3,4,5.$
\item $4s_4 \geq s_2^2$, and
\item $12s_5 - 5s_2s_3 +5s_3\sqrt{4s_4 - s_2^2} \geq 0$
\end{enumerate}
\end{theorem}
Following this, Laffey and Šmigoc  gave a generalized necessary and sufficient condition for the realizability of a list of $n$ complex numbers in which every eigenvalue except the Perron eigenvalue is complex and has a non-positive real part. The following lemma, which was central to their analysis, is recalled below.
\begin{lemma}\cite{lasmi}\label{lemma}
Consider a non-negative real number $t$ and let $\delta_2,\delta_3,\ldots,\delta_n$ be complex numbers with real parts less than or equal to zero, such that the list $(\delta_2,\delta_3,\ldots,\delta_n)$ is closed under complex conjugation. Let $\rho =2t-\delta_2-\cdots-\delta_n,$ and
\begin{center}
  $\displaystyle f(x)=(x-\rho)\prod_{j=2}^{n}(x-\delta_j)=x^n -2tx^{n-1}+b_2x^{n-2}+\cdots+b_n,$ then
\end{center}
$b_2 \leq 0$ implies $b_j \leq 0$ for $j=3,4,\ldots,n.$
\end{lemma}
\begin{theorem}\cite{lasmi}\label{realising five}
	Let $\delta_1, \delta_2,\ldots,\delta_n$ be complex numbers with real parts less than or equal to zero and
	let $\rho$ be a positive real number. Then the list $\sigma= \{\rho,\delta_1, \delta_2,\ldots,\delta_n\}$ is the spectrum of a non-negative matrix if and only if the following conditions are satisfied:
	\begin{enumerate}[(i)]
		\item The list $(\rho,\delta_1, \delta_2,\ldots,\delta_n)$ is closed under complex conjugation.
		\item $s_1=\rho+\sum_{i=2}^{n} \delta_i \geq 0$
		\item $s_2=\rho+\sum_{i=2}^{n} {\delta_i}^2 \geq 0$
		\item $s_1^2 \leq ns_2$
	\end{enumerate}
\end{theorem}
Ana I. Julio and Ricardo L. Soto were the first to introduce the persymmetric non-negative inverse eigenvalue problem (PNIEP). They established sufficient conditions for the existence of solutions to the PNIEP by applying Rado’s theorem \cite{AnRi}.The persymmetric realizability problem was subsequently investigated by A. I. Julio et al. in \cite{Ana}, where it was shown that the non-negative inverse eigenvalue problem (NIEP) and the PNIEP are equivalent for any realizable list of three complex numbers.
\begin{theorem}\cite{Ana}
Let $\Delta:= \{\delta_1, \delta_2,\delta_3\}$  be a realizable list of complex numbers. Then $\Delta$ is realizable by a persymmetric matrix.
\end{theorem}
They proved that, for certain realizable lists of four complex numbers, the NIEP and the PNIEP coincide, however, in the general case, the two problems are distinct. Furthermore, authors established a sufficient condition under which a trace-zero list of five complex numbers can be realized as the spectrum of a persymmetric non-negative matrix, given by
\begin{theorem}\cite{Ana}
Let $\Delta:= \{\delta_1, \delta_2,\ldots,\delta_5\}$ be a realizable list of complex numbers with $s_1=0.$ If $12s_5-5s_2s_3 \geq 0$, then $\Delta$ is persymmetrically realizable.
\end{theorem}
Following these results authors raise the question,  "whether a realizable list of complex numbers satisfying $s_1=0$ and $12s_5-5s_2s_3 < 0$, is persymmetrically realizable or not?",  they also mention that there exists a non-negative matrix $\begin{bmatrix}
	0 & 1& 0& 0& 0\\
	5 &0& 1& 0& 0\\
	0 &0& 0& 1& 0\\
	0 &0 &0& 0& 1\\
	30& 86& 30& 0 &0\\
\end{bmatrix}$
 which realizes the set $\Delta:=\{4,1,-3,-1+3i,-1-3i\}$ which satisfies $s_1 =0 $ and $12s_5 - 5s_2s_3 < 0$.
But the existence of a non-negative persymmetric matrix in this case remains unknown. This question is addressed and resolved in the following section.
\section{Main Results}
This section presents the main results of this study. In particular, a complete solution to the open problem described in the section \ref{sec:NIEP} is provided. Three persymmetric matrices which realizes the set $\Delta_1:=\{4,1,-3,-1+3i,-1-3i\}$ are:

$\begin{bmatrix}
0&0&0&0&43-2\sqrt{451}~\\
\frac{43+2\sqrt{451}}{15}&0&0&1&0\\
2\sqrt{\frac{86+4\sqrt{451}}{15}}&0&0&0&0\\
1&5&0&0&0\\
0&1&2\sqrt{\frac{86+4\sqrt{451}}{15}}&\frac{43+2\sqrt{451}}{15}&0
\end{bmatrix},\,\begin{bmatrix}
0&0&0&0&1\\
\sqrt{\frac{45}{2}}&0&0&1&0\\
\sqrt{30}&0&0&0&0\\
0&4&0&0&0\\
1&0&\sqrt{30}&\sqrt{\frac{45}{2}}&0
\end{bmatrix}$\\

and $\begin{bmatrix}
0&0&0&0&43+2\sqrt{451}~\\
\frac{43-2\sqrt{451}}{15}&0&0&1&0\\
2\sqrt{\frac{86- 4\sqrt{451}}{15}}&0&0&0&0\\
1&5&0&0&0\\
0&1&2\sqrt{\frac{86-4\sqrt{451}}{15}}&\frac{43- 2\sqrt{451}}{15}&0
\end{bmatrix}.$\\

These examples posses the following persymmetric forms 
\begin{center}
$ \begin{bmatrix}
0&0&0&0&a\\
b&0&0&1&0\\
c&0&0&0&0\\
1&d&0&0&0\\
0&1&c&b&0
\end{bmatrix}$,  $\begin{bmatrix}
0&0&0&0&a\\
b&0&0&1&0\\
c&0&0&0&0\\
0&d&0&0&0\\
1&0&c&b&0
\end{bmatrix}$
\end{center}
 where a,b,c,d are non-negative. The following theorem make use of these persymmetric forms to establish the equivalence between the NIEP and the PNIEP for realizable lists of five complex numbers with zero sum.
\begin{theorem}\label{main theorem}
Let $\Delta := \{\delta_1,\delta_2,\delta_3,\delta_4,\delta_5\}$ be a realizable list of five complex numbers with $s_1=0.$ Then $\Delta $ is persymmetricallly realizable.
\begin{prf}
Since $\Delta$ is a realizable list of complex numbers, it satisfies all the conditions of Theorem \ref{realizing three}. The proof is divided into three cases as follows.

\begin{enumerate}[(a)]
\item \textit{Case 1}: $\Delta$ satisfies $ 6s_5 - 5s_2s_3 \geq 0.$ Let $A$ be a persymmetric matrix of the form
$A:= \begin{bmatrix}
0&1&0&0&0\\
a&0&0&1&0\\
b&0&0&0&0\\
c&0&0&0&1\\
d&c&b&a&0
\end{bmatrix}.$
The characteristic polynomial of $A$ is
\begin{center}
$f(z)=z^5-2az^3 - 2cz^2 +(a^2-d)z-b^2.$  
\end{center}
Application of Newton’s identities yields equations relating $a,b,c,d$ to $s_k$'s which are given by
\[2a = \frac{s_2}{2},\quad
2c = \frac{s_3}{3},\quad
a^2 - d= \frac{s_2^2}{8} - \frac{s_4}{4}\quad \textit{and} \quad
-b^2 = \frac{s_2s_3}{6} - \frac{s_5}{5}.
\]
Solve these equations for the values of $a,b,c,d$  in terms of the $s_k$'s as 
\[a= \frac{s_2}{4},\quad
b=\frac{\sqrt{30}\sqrt{6s_5-5s_2s_3}}{30},\quad
c= \frac{s_3}{6}\quad \textit{and} \quad
d=\frac{4s_4-s_2^2}{16},
\]
and hence the non-negative matrix is
\begin{center}
$A= \begin{bmatrix}
0&1&0&0&0\\
\frac{s_2}{4}&0&0&1&0\\
\frac{\sqrt{30}\sqrt{6s_5-5s_2s_3}}{30}&0&0&0&0\\
\frac{s_3}{6}&0&0&0&1\\
\frac{4s_4-s_2^2}{16}&\frac{s_3}{6}&\frac{\sqrt{30}\sqrt{6s_5-5s_2s_3}}{30}&\frac{s_2}{4}&0
\end{bmatrix}.$
\end{center} 
By hypothesis every entry in $A$ is non-negative. Hence $A$ realizes $\Delta$ persymmetrically.
\item \textit{Case 2}: $\Delta$ satisfies $6s_5 - 5s_2s_3 < 0$ and $2s_4 - s_2^2 \geq 0.$ The proof splits into two parts.
\begin{enumerate}[(i)]
\item Let $s_5=0, s_2 \neq 0, s_3 \neq 0,$ and  $B:= \begin{bmatrix}
	0&0&0&0&1\\
	a&0&0&1&0\\
	b&0&0&0&0\\
	0&c&0&0&0\\
	0&0&b&a&0 
\end{bmatrix}$
be a persymmetric matrix with the characteristic polynomial 
$f(z)=z^5 -cz^3-b^2z^2-a^2cz+b^2c.$ Applying (\ref{newton}) provide the relations
\[c=\frac{s_2}{2},\quad
b^2=\frac{s_3}{3},\quad
a^2c= \frac{s_4}{4}-\frac{s_2^2}{8}\quad \textit{and} \quad
b^2c= \frac{s_2s_3}{6}-\frac{s_5}{5}.
\]
Solving the system for $a,b$ and $c$ yields
$a=\sqrt{(\frac{s_4}{4}-\frac{s_2^2}{8})\frac{2}{s_2}}=\sqrt{\frac{2s_4 -s_2^2}{4s_2}},\quad
b=\sqrt{\frac{s_3}{3}}\quad and \quad
c=\frac{s_2}{2}.$
Hence, the realizing matrix is obtained as
\begin{center}
$B=\begin{bmatrix}
0&0&0&0&1\\
\sqrt{\frac{2s_4 -s_2^2}{4s_2}}&0&0&1&0\\
\sqrt{\frac{s_3}{3}}&0&0&0&0\\
0&\frac{s_2}{2}&0&0&0\\
0&0&\sqrt{\frac{s_3}{3}}&\sqrt{\frac{2s_4 -s_2^2}{4s_2}}&0
\end{bmatrix}$
\end{center} 
\item Let $s_5 \neq 0, s_2 \neq 0 , s_3 \neq 0.$ In this case, consider the persymmetric matrix of the form 
$ C:= \begin{bmatrix}
0&0&0&0&a\\
b&0&0&1&0\\
c&0&0&0&0\\
0&d&0&0&0\\
1&0&c&b&0    
\end{bmatrix}$
with characteristic polynomial
$f(z)=z^5+(-a-d)z^3-ac^2z^2+(-ab^2d+ad)z+ac^2d.$ (\ref{newton}) gives the following relations between the coefficients of $f(z)$ and the $k^{th}$ moments $s_k.$
\[a+d=\frac{s_2}{2},\quad
ac^2=\frac{s_3}{3},\quad
ad-ab^2d= \frac{s_2^2}{8}- \frac{s_4}{4}\quad \textit{and}\quad
ac^2d=\frac{s_2s_3}{6}-\frac{s_5}{5}.
\]
Thus the values $a,b,c$ and $d$ are
\begin{align*}
a&= \frac{3s_5}{5s_3},\,\,
b=\frac{\sqrt{3}}{6}\sqrt{\frac{12s_5(5s_2s_3-6s_5)+25s_3^2(2s_4-s_2^2)}{s_5(5s_2s_3-6s_5)}},\,\,
c=\frac{\sqrt{5}s_3}{3\sqrt{s_5}}\quad \textit{and}\\
d&= \frac{5s_2s_3-6s_5}{10s_3}.
\end{align*}
\end{enumerate}
Then the matrix $C$ is
\begin{center}
$\begin{bmatrix}
0&0&0&0&\frac{3s_5}{5s_3}\\
\frac{\sqrt{3}}{6}\sqrt{\frac{12s_5(5s_2s_3-6s_5)+25s_3^2(2s_4-s_2^2)}{s_5(5s_2s_3-6s_5)}}&0&0&1&0\\
\frac{\sqrt{5}s_3}{3\sqrt{s_5}}&0&0&0&0\\
0&\frac{5s_2s_3-6s_5}{10s_3}&0&0&0\\
1&0&\frac{\sqrt{5}s_3}{3\sqrt{s_5}}&\frac{\sqrt{3}}{6}\sqrt{\frac{12s_5(5s_2s_3-6s_5)+25s_3^2(2s_4-s_2^2)}{s_5(5s_2s_3-6s_5)}}&0 
\end{bmatrix}.$
\end{center}
Under the given assumptions the matrix $C$ persymmetrically realizes $\Delta.$
\item \textit{Case 3}: 
$6s_5 - 5s_2s_3 < 0$ and  $2s_4-s_2^2 < 0.$ Consider a persymmetric matrix of the  form,
$D:=\begin{bmatrix}
0&0&1&0&0\\
0&0&0&1&0\\
c&0&0&0&1\\
b&d&0&0&0\\
a&b&c&0&0    
\end{bmatrix},$ with characteristic polynomial
$f(z)=z^5+(-2c-d)z^3-az^2+2cdz+ad-b^2.$
Newton’s identities yield the following relations between the coefficients of $f(z)$ and $s_k.$
\[2c+d=\frac{s_2}{2},\quad
a=\frac{s_3}{3},\quad
2cd= \frac{s_2^2}{8}- \frac{s_4}{4}\quad \textit{and}\quad
ad-b^2=\frac{s_2s_3}{6}-\frac{s_5}{5},
\]
and the values of $a,b,c,d$ are derived as 
\begin{align*}
a&=\frac{s_3}{3},\,
b=\sqrt{\frac{12s_5 - 5s_2s_3 + 5 s_3\sqrt{4s_4-s_2^2}}{60}},\,
c=\frac{s_2 -\sqrt{4s_4-s_2^2}}{8}\quad \textit{and}\\
d&=\frac{s_2 +\sqrt{4s_4-s_2^2}}{4}.
\end{align*}
Therefore the realizing matrix is
\begin{center}
$D=\begin{bmatrix}
0&0&1&0&0\\
0&0&0&1&0\\
\frac{s_2 -\sqrt{4s_4-s_2^2}}{8}&0&0&0&1\\
\sqrt{\frac{12s_5 - 5s_2s_3 + 5 s_3\sqrt{4s_4-s_2^2}}{60}}&\frac{s_2 +\sqrt{4s_4-s_2^2}}{4}&0&0&0\\
\frac{s_3}{3}&\sqrt{\frac{12s_5 - 5s_2s_3 + 5 s_3\sqrt{4s_4-s_2^2}}{60}}&\frac{s_2 -\sqrt{4s_4-s_2^2}}{8}&0&0    
\end{bmatrix}.$
\end{center}
All entries of $D$ are non-negative by hypothesis. Thus $D$ realizes $\Delta$ persymmetrically.\\
Hence, the equivalence between the NIEP and the PNIEP is established. $\hfill{\square}$
\end{enumerate}
\end{prf}
\end{theorem}
The following examples illustrates the theorem.
\begin{exam}
Consider the set $\Delta_1:= \{2,-1+i,-1-i,i,-i\}$. The first five moments are, 
$s_1=0,\, s_2 =2,\,s_3=12,\, s_4=10,\, s_5=40.$
Also we have 
\[4s_4 - s_2^2 = 36,\quad 6s_5 - 5s_2s_3 = 120 >0.\]
Then by case 1 of Theorem \ref{main theorem}, the realizing matrix is
$\begin{bmatrix}
		0&1&0&0&0\\
		\frac{1}{2}&0&0&1&0\\
		2&0&0&0&0\\
		2&0&0&0&1\\
		\frac{9}{4}&2&2&\frac{1}{2}&0
\end{bmatrix}$
\end{exam}
\begin{exam}
The set $\Delta_2:=\{6,-5,1,-1+4i,-1-4i\}$ and the first five moments are
$s_1=0,\, s_2 =32,\, s_3=186,\, s_4=2244,\, s_5=2410.$
Note that,
$2s_4 - s_2^2 = 3464 > 0,\, 6s_5 - 5s_2s_3 = -15300 < 0,\,
12s_5 - 5s_2s_3= -840 < 0.$
By case 2 of Theorem \ref{main theorem} the realizing matrix is
$\begin{bmatrix}
0&0&0&0&\frac{241}{31}\\
\frac{\sqrt{1910272}}{2\sqrt{61455}}&0&0&1&0\\
\frac{62}{\sqrt{482}}&0&0&0&0\\
0&\frac{255}{31}&0&0&0\\
1&0&\frac{62}{\sqrt{482}}&\frac{\sqrt{1910272}}{2\sqrt{61455}}&0
\end{bmatrix}$
\end{exam}
\begin{exam}
The first five moments corresponds to the set $\Delta_3:= \{6,4,-4,-3+4i,-3-4i\}$ are
$s_1=0,\, s_2 =54,\, s_3=450,\, s_4=754,\,\text{and}\,  s_5=8250.$
Also, note that $2s_4 - s_2^2 = -1408 < 0,\,6s_5 - 5s_2s_3 =-72000 < 0,\,
12s_5 - 5s_2s_3=-22500< 0.$
The realizing matrix of $\Delta_3$ by case 3 of Theorem \ref{main theorem} is
$\begin{bmatrix}
0&0&1&0&0\\
0&0&0&1&0\\
\frac{11}{2}&0&0&0&1\\
0&16&0&0&0\\
150&0&\frac{11}{2}&0&0
\end{bmatrix}$
\end{exam}
\section{Sufficient Conditions for Persymmetric realizability} 
\subsection{\textbf{Sufficient Conditions for Persymmetric Realizability of Spectra with Single Positive Eigenvalue}}
This section analyzes the persymmetric realizability of the set of complex numbers consisting of one element with positive real part (Perron eigenvalue) and all remaining elements with non-positive real parts. Theorem \ref{realising five} establish the non-negative realizability of these spectra.  The following theorem provides a sufficient condition for the persymmetric realizability of sets with these criteria of order $5.$
\begin{theorem}\label{persymm realiz}
Let $\{\delta_1,\delta_2,\delta_3,\delta_4\}$ be a list of complex numbers that is closed under conjugation and having non positive real part and $\delta$ be a positive real number. Then the list $\Delta:= \{\delta,\delta_1,\delta_2,\delta_3,\delta_4\}$ is persymmmetrically realizable if $s_1 \geq 0$ and $s_2-s_1^2 \geq 0.$
\begin{prf}
Suppose the set $\Delta$ is the spectrum of some persymmetric matrix $A,$ then its characteristic polynomial is given by, $p(z)=(z-\delta)(z-\delta_1)(z-\delta_2)(z-\delta_3)(z-\delta_4).$
Consider the  matrices $A_1, A_2, A_3, A_4,$ and $A_5$ with the following persymmetric structures, 
\begin{center}
	$A_1=\begin{bmatrix}
		0&1&0&0&0\\
		p&0&1&0&0\\
		q&0&t&1&0\\
		r&0&0&0&1\\
		s&r&q&p&0
	\end{bmatrix}$, 	$A_2=\begin{bmatrix}
	0&1&0&0&0\\
	0&0&1&0&0\\
	q&p&t&1&0\\
	r&0&p&0&1\\
	s&r&q&0&0
	\end{bmatrix}$,     $A_3=\begin{bmatrix}
	0&1&0&0&0\\
	p&0&1&0&0\\
	q&0&t&1&0\\
	r&q&0&0&1\\
	s&r&q&p&0
	\end{bmatrix}$,      $A_4=\begin{bmatrix}
	0&1&0&0&0\\
	p&0&1&0&0\\
	q&p&t&1&0\\
	r&q&p&0&1\\
	s&r&q&p&0
	\end{bmatrix}$,       $A_5=\begin{bmatrix}
0&1&0&0&0\\
0&0&1&0&0\\
0&p&t&1&0\\
r&q&p&0&1\\
s&r&0&0&t
\end{bmatrix}$ \quad where $p,q,r,s,t \in \mathbb{R}.$
\end{center}
\begin{enumerate}[(i)]
\item The characteristic polynomial of $A_1$ is 
\[f(z)= z^5-tz^4-2pz^3 +(2pt-2q)z^2+(p^2-2r)z+(-p^2 t+2pq-s).\]
The constants $p,q,r,s,t$ are expressed in terms of $s_i$ for $i=1,2\ldots5$ using Newton's identities are as follows
\begin{equation}\label{A_1p,q,r}
p=\frac{s_2-s_1^2}{4},\quad q=\frac{s_3-s_1^3}{6},\quad r=\frac{s_4+6s_1^2s_2-16s_1s_3-3s_2^2+12s_4}{96}
\end{equation}
\begin{equation}\label{A_1s,t}
s=\frac{7s_1^5-10s_1^3s_2+6s_1^2s_2+20s_1^2s_3+9s_2^2s_1-60s_1s_4-20s_2s_3+48s_5}{240},\quad t=s_1.
\end{equation}
The constants $p,q,r,s$ and $t$ can also be obtained in terms of members of  $ \Delta$ by comparing coefficients of $p(z)$ and $f(z).$ It remains to show that each of these constants are non-negative. Since  $s_1\geq 0 $ and $  s_2-s_1^2 \geq 0.$ From (\ref{A_1p,q,r}) and (\ref{A_1s,t}) we have,
$ s_1 \geq 0 \implies t \geq 0$ and 
$s_2-s_1^2\geq 0 \implies p\geq  0.$
Thus the Lemma \ref{lemma} applies to the polynomial $f(z)$ and the following holds.
\[2pt-2q \leq 0 \implies pt \leq q \implies q\geq 0\]
\[p^2-2r \leq 0 \implies p^2 \leq 2r \implies r \geq 0\]
\[2pq-p^2t-s \leq 0 \implies s \geq 2p(q-pt) \implies s\geq 0.\] 
Therefore all these constants are non-negative.

\item The characteristic polynomial of $A_2$ is $f(z)= z^5-tz^4-2pz^3 -2qz^2 -2rz-s.$ The constants $p,q,r,s,t$ are expressed in terms of $s_i$ for $i=1,2\ldots5$ utilizing (\ref{newton}) are as follows
\begin{equation}\label{A_2p,q,r}
	p=\frac{s_2-s_1^2}{4},\quad q=\frac{s_1^3+2s_3-3s_1s_2}{12},\quad r=\frac{6s_4+6s_1^2s_2-3s_2^2-s_1^4-8s_1s_3}{48}
\end{equation}
\begin{equation}\label{A_2s,t}
	s=\frac{ 24s_5 -30s_1s_4 20(s_2s_3 -s_3s_1^2) - 10s_2s_1^3 + 15s_1s_2^2 + s_1^5}{120},\quad t=s_1
\end{equation}
Since  $s_1\geq 0 $ and $  s_2-s_1^2 \geq 0.$ From (\ref{A_2p,q,r}) and (\ref{A_2s,t}) we have,
$ s_1 \geq 0 \implies t \geq 0$ and 
$s_2-s_1^2\geq 0 \implies p\geq  0.$
Thus the Lemma \ref{lemma} applies to the polynomial $f(z)$ and the following holds.
\[-2q \leq 0 \implies q \geq 0, \,\, -2r \leq 0 \implies r \geq 0,\,\, -s \leq 0 \implies s \geq 0.\]
Therefore, all these constants are non-negative.\\
\item The characteristic polynomial of $A_3$ is 
\[f(z)= z^5-tz^4-2pz^3 +(2pt-3q)z^2 +(p^2-2r)z-p^2t+2pq-s\]
The constants $p,q,r,s,t$ expressed in terms of $s_i$ for $i=1,2\ldots5$ using (\ref{newton}) are as follows
\begin{equation}\label{A_3p,q,r}
	p=\frac{s_2-s_1^2}{4},\quad q=\frac{s_3-s_1^3}{9},\quad r=\frac{s_1^4+6s_2s_1^2-16s_1s_3-3s_2^2+12s_4}{96}
\end{equation}
\begin{equation}\label{A_3s,t}
	s=\frac{s_1^5-10s_1^3s_2+80s_1^2s_3+45s_1s_2^2-180s_1s_4-80s_2s_3+144s_5}{720},\quad t=s_1.
	\end{equation}
Since  $s_1\geq 0 $ and $  s_2-s_1^2 \geq 0.$ From (\ref{A_3p,q,r}) and (\ref{A_3s,t}) we have,
$ s_1 \geq 0 \implies t \geq 0$ and 
$s_2-s_1^2\geq 0 \implies p\geq  0.$
Hence, the Lemma \ref{lemma} applies to the polynomial $f(z)$ and the following holds,
\[2pt-3q \leq 0 \implies 2pt \leq 3q \implies q \geq 0,\,\, p^2-r \leq 0 \implies p^2 \leq r \implies r \geq 0, \]
\[-p^2t+2pq-s \leq 0 \implies s\geq p(2q-pt) \implies s\geq 0.\]
This shows that all constants are non negative.\\
\item The characteristic polynomial of $A_4$ is 
\[f(z)= z^5-tz^4-4pz^3 +(2pt-3q)z^2 +(3p^2-2r)z-p^2t+2pq-s\]
The constants $p,q,r,s,t$ are expressed in terms of $s_i$ for $i=1,2\ldots5$ using (\ref{newton}) are as follows
\begin{equation}\label{A_4p,q,r}
	p=\frac{s_2-s_1^2}{8},\quad q=\frac{4s_3-3s_1s_2-s_1^3}{36},\quad r=\frac{s_1^4+30s_1^2s_2-64s_1s_3-15s_2^2+48s_4}{384}
\end{equation}
\begin{equation}\label{A_4s,t}
	s=-\frac{s_1^5+110s_1^3s_2-400s_1^2s_3-255s_1s_2^2+720s_1s_4+400s_2s_3-576s_5}{2880},\quad t=s_1
	\end{equation}
Since  $s_1\geq 0 $ and $  s_2-s_1^2 \geq 0.$ From (\ref{A_4p,q,r}) and (\ref{A_4s,t}) we have,
$ s_1 \geq 0 \implies t \geq 0$ and 
$s_2-s_1^2\geq 0 \implies p\geq  0.$
Thus the Lemma \ref{lemma} applies to the polynomial $f(z)$ and the following holds.
\[2pt-3q \leq 0 \implies 3q \geq 2pt \implies q \geq 0, \,\,
3p^2-r \leq 0 \implies 3p^2 \leq r \implies r \geq 0,\]
\[-p^2t+2pq-s \leq 0 \implies s\geq p(2q-pt) \implies s\geq 0.\]
This clearly implies all constants are non-negative.
\item The characteristic polynomial of $A_5$ is 
$ f(z)= z^5-tz^4-2pz^3 -qz^2 -2rz-s. $
The constants $p,q,r,s,t$ expressed in terms of $s_i$ for $i=1,2\ldots5$ using (\ref{newton}) are as follows
\begin{equation}\label{A_5p,q,r}
	p=\frac{s_2-s_1^2}{8},\quad q=\frac{s_1^3+2s_3-3s_1s_2}{6},\quad r=\frac{6s_4+6s_1^2s_2-3s_2^2-s_1^4-8s_1s_3}{48}
	\end{equation}
	\begin{equation}\label{A_5s,t}
	s=\frac{ 24s_5 -30s_1s_4 20(s_2s_3 -s_3s_1^2) - 10s_2s_1^3 + 15s_1s_2^2 + s_1^5}{120},\quad t=s_1.
	\end{equation}
Since  $s_1\geq 0 $ and $  s_2-s_1^2 \geq 0.$ From (\ref{A_5p,q,r}) and (\ref{A_5s,t}) we have,
$ s_1 \geq 0 \implies t \geq 0$ and 
$s_2-s_1^2\geq 0 \implies p\geq  0.$
Thus the Lemma \ref{lemma} applies to the polynomial $f(z)$ and the following holds, 
\[-q \leq 0 \implies q \geq 0,\,\,
-2r \leq 0 \implies r \geq 0,\,\, 
-s \leq 0 \implies s \geq 0.\]
\end{enumerate}
Since  all constants $p,q,r,s$ and $t$ are non-negative, a non-negative persymmetric matrix is possible with the structures of $A_1, A_2, A_3, A_4, A_5$ which realizes $\Delta.$  $\hfill{\square}$
\end{prf}
\end{theorem}
\begin{Rem}
If $s_1=0,$ the spectrum $\Delta$ admits a non-negative Toeplitz realization with the matrix structure $A_4.$ 
\end{Rem}
\begin{Rem}
Since the conditions are not necessary for a persymmetric realizability the converse implication fails. For example, 
	\begin{itemize}
		\item [(i)] Consider the set $\Delta:=\{10,0,-1,-1+6i,-1-6i\}.$ For this choice of $\Delta$, $s_2 - s_1^2 = -18 < 0$. Nevertheless, a persymmetric realization exists, as follows::
		\begin{center}
			$\begin{bmatrix}
				\frac{7}{5}&1&0&0&0\\
				\frac{53}{10}&\frac{7}{5}&1&0&0\\
				\frac{4626}{25}&0&\frac{7}{5}&1&0\\
				\frac{695597}{1000}&0&0&\frac{7}{5}&1\\
				\frac{9900378}{3125}&\frac{695597}{1000}&\frac{4626}{25}&\frac{53}{10}&\frac{7}{5}
			\end{bmatrix}$
		\end{center} 
		\item [(ii)] Consider the set $\Delta : =\{9,-1+4i,-1-4i,-1+4i,-1-4i\}.$ Then,   $s_2-s_1^2 = -4 <0, 
		$ yet there exists a nonnegative persymmetric matrix that realizes $\Delta,$ given as:
		\begin{center}
			$\begin{bmatrix}
				1&1&0&0&0\\
				4&1&1&0&0\\
				144&0&1&1&0\\
				448&0&0&1&1\\
				4352&448&144&4&1
			\end{bmatrix}$
		\end{center}
	\end{itemize} 
\end{Rem}
\begin{Rem}
Consider the set $\Delta:=\{\delta,a+ib,a-ib,c+id,c-id\}.$ Suppose $a, c \leq 0$ are given; then a value of $\delta$ can be chosen such that
$s_1=\delta+2(a+c) \geq 0$ i.e., $\delta \geq -2(a+c).$
From $s_2-s_1^2 \geq 0,$ a bound for $b^2+d^2$ is obtained as follows:
\[s_2-s_1^2 \geq 0 \implies \delta^2+2(a^2-b^2+c^2-d^2)-(\delta+2a+2c)^2 \geq 0\]
i.e., \[0 \leq b^2+d^2 \leq -2\delta(a+c)-4ac-(a^2+c^2).\]
\end{Rem}
 
\begin{exam}
Given $a=-1,\, c= -2.$ Since $s_1 \geq 0$, any $\delta$ satisfying $\delta \geq -2(a + c)$ can be chosen. For each such $\delta$, the bound for $b^2 + d^2$ can be determined from the inequality $s_2 - s_1^2 \geq 0$.
If $\delta= -2(a+c)=6$ then,
\[0 \leq b^2+d^2 \leq -2\delta(a+c)-4ac-(a^2+c^2).\]
i.e., $0 \leq b^2+d^2 \leq 23.$
Take $b=3$ and $d=2.$ Then, the persymmetric matrix realizing the set $\{6,-1+3i,-1-3i,-2+2i,-2-2i\}$ is
\begin{center}
$\begin{bmatrix}
0&1&0&0&0\\
5&0&1&0&0\\
50&0&0&1&0\\
\frac{281}{2}&0&0&0&1\\
980&\frac{281}{2}&50&5&0
\end{bmatrix}.$
\end{center} 
\end{exam}
The next objective is to generalize Theorem \ref{persymm realiz}. Ricardo L. Soto et.al. in \cite{AJ-RS-MS}, obtains the conditions for the existence of an $n\times n$ entrywise non-negative matrix with prescribed elementary divisors. In particular, they show that the companion matrix of the polynomial $\displaystyle p(z)=x^n-\sum_{k=1}^{n}c_{k-1}x^{n-k},$ is similar to a persymmetric matrix given by,
\begin{center}
	$P=\begin{bmatrix}
		p_0&1&0&0&0&0&\ldots&0&0\\
		p_1&p_0&1&0&0&0&\ldots&0&0\\
		p_2&0&p_0&1&0&0&\ldots&0&0\\
		\vdots&\vdots&\vdots&\vdots&\ddots&\vdots&\ldots&\vdots&\vdots\\
		p_{\frac{n-1}{2}}&0&0&\ldots&p_0&1&\ldots&0&0\\
		\vdots&\vdots&\vdots&\vdots&\vdots&\vdots&\ddots&\vdots&\vdots\\
		p_{n-3}&0&0&\cdots&0&0&\ldots&1&0\\
		p_{n-2}&0&0&\cdots&0&0&\ldots&p_0&1\\
		p_{n-1}&p_{n-2}&p_{n-3}&\cdots&p_{\frac{n-1}{2}}&p_{\frac{n+1}{2}}&\ldots&p_1&p_0
	\end{bmatrix}.$
\end{center}
An interesting result in their work is, the non-negativity of the companion matrix implies the non-negativity of the persymmetric matrix. This result motivates us to generalize our theorem to the case of arbitrary order $n.$
\begin{theorem}
Let $\{\delta_2,\delta_3,\ldots,\delta_n\}$ be a list of complex numbers is closed under conjugation and having non positive real part and $\delta$ be a positive real number. Then the list $\Delta:= \{\delta,\delta_2,\delta_3,\ldots,\delta_n\}$ is persymmmetrically realizable if $s_1 \geq 0$ and $s_2-s_1^2 \geq 0.$
\end{theorem}
\begin{prf}
Suppose $\Delta$ is the spectrum of a persymmetric matrix $P,$ with the characteristic polynomial  $p(z)=z^n-a_{1}z^{n-1}-a_{2}z^{n-2}-\cdots-a_n.$
The companion matrix $C$ corresponding to $p(z)$ is 
\begin{center}
$C=\begin{bmatrix}
0&1&0&\cdots&0\\
0&0&1&\cdots&0\\
\vdots&\vdots&\vdots&\ddots&\vdots\\
0&0&0&\cdots&1\\
a_{n}&a_{n-1}&a_{n-2}&\cdots&a_1
\end{bmatrix}.$
\end{center}
The coefficients of $p(z)$ is obtained in terms of $s_i$ using (\ref{newton}) as,
 $a_{1}=s_1, \,\, a_{2}=\frac{s_2-s_1^2}{2}.$
 Since the hypothesis implies the non-negativity of both $a_{n-1}$ and $a_{n-2}$, apply Lemma \ref{lemma} on $p(z)$ to obtain $a_{i} \geq 0 \quad \text{for}\quad i=3,4,\ldots,n.$ Thus the companion matrix $C$ is non-negative. Therefore corresponding persymmetric matrix $P$ is also non-negative.  $\hfill{\square}$
\end{prf}
\subsection{\textbf{Sufficient Conditions for Persymmetric Realizability of Five-Element Spectra}}
This section focuses on establishing sufficient conditions for the realization of a list of five complex numbers by a non-negative persymmetric matrix. The following theorem states that certain criteria must be met to achieve such a non-negative persymmetric realization. 
\begin{theorem}\label{suffi condtn}
Let $\Delta := \{\delta_1, \delta_2, \ldots, \delta_5\},$ which is closed under complex conjugation. Then $\Delta$ is realizable by a non-negative persymmetric matrix if the following conditions are satisfied:
\begin{enumerate}[(a)]
\item $s_1\geq 0$
\item $5s_2-s_1^2 \geq 0$
\item $2s_1^3 - 15s_1 s_2 + 25s_3 \geq 0$
\item $ 2000 s_4+ 730 s_1^2 s_2 -73s_1^4 - 1600 s_1 s_3 - 625 s_2^2  \geq 0$
\item $394s_1^5 - 4925s_1^3 s_2 + 12125s_1^2 s_3 + 9375s_1 s_2^2
- 22500s_1 s_4 - 15625s_2 s_3 + 22500s_5 \geq 0$
\end{enumerate}
\begin{prf}
Let $A$ be the matrix associated with the spectrum $\Delta$, given by
\begin{center}
$A= \begin{bmatrix}
t&1&0&0&0\\
p&t&1&0&0\\
q&p&t&1&0\\
r&q&p&t&1\\
s&r&q&p&t
\end{bmatrix}$ where $p,q,r,s,t \in \mathbb{R}$
\end{center}
and the characteristic polynomial is $f(z)= z^5- 5tz^4+(-4p+10t^2)z^3+(12pt-3q-10t^3)z^2+(3p^2-12pt ^2+6qt-2r+5t^4)z+(2pq-3p^2t+4pt^3-3qt^2+2rt-s-t^5).$
The constants $p,q,r,s,t$ are expressed in terms of $s_i$ for $i=1,2\ldots5$ using (\ref{newton}) are as follows,
\begin{equation*}\label{A_p,q,r}
	p=\frac{5s_2-s_1^2}{40},\, q=\frac{2s_1^3 - 15s_1 s_2 + 25s_3}{225},\, r=\frac{2000 s_4+ 730 s_1^2 s_2 -73s_1^4 - 1600 s_1 s_3 - 625 s_2^2}{16000}
\end{equation*}
\begin{equation*}\label{A_s,t}
	s=\frac{394s_1^5 - 4925s_1^3 s_2 + 12125s_1^2 s_3 + 9375s_1 s_2^2
		- 22500s_1 s_4 - 15625s_2 s_3 + 22500s_5}{112500},\, t=\frac{s_1}{5}.
\end{equation*}
The non-negativity of the above constants ensures the non-negativity of the matrix. It is also  observed that the matrix is both persymmetric and Toeplitz. Hence, the proof is complete. $\hfill{\square}$
\end{prf}
\end{theorem}
\begin{exam}
Let $\Delta:=\{6,1+2i,1-2i,3+2i,3-2i\}.$ Then the corresponding non-negative persymmetric realizing matrix is given by,
\begin{center}
$A= \begin{bmatrix}
	\frac{14}{5}&1&0&0&0\\
	\frac{1}{10}&\frac{14}{5}&1&0&0\\
	\frac{496}{75}&\frac{1}{10}&\frac{14}{5}&1&0\\
	\frac{4027}{1000}&\frac{496}{75}&\frac{1}{10}&\frac{14}{5}&1\\
	\frac{889888}{9375}&\frac{4027}{1000}&\frac{496}{75}&\frac{1}{10}&\frac{14}{5}
\end{bmatrix}$
\end{center}
\end{exam}
\begin{exam}
Let $\Delta:=\{8,2,1,2+2i,2-2i\}.$ Then corresponding non-negative persymmetric realizing matrix is,
\begin{center}
$A= \begin{bmatrix}
		3&1&0&0&0\\
		3&3&1&0&0\\
		\frac{46}{3}&3&3&1&0\\
		56&\frac{46}{3}&3&3&1\\
		142&56&\frac{46}{3}&3&3
	\end{bmatrix}$
\end{center}
\end{exam}
\begin{Rem}
The conditions stated in Theorem \ref{suffi condtn} are sufficient but not necessary. For example, consider $\Delta:=\{4,2,-3,-1+3i,-1-3i\}.$ $\Delta$ is not realizable by a matrix of the form stated in the theorem, since the value of  $s=-\frac{614896}{3125}<0.$ However, it is realized by another persymmetric matrix given by:
\begin{center}
$\begin{bmatrix}
0&0&0&0&1\\
0&0&0&1&0\\
4\sqrt{2}&0&1&0&0\\
\frac{15\sqrt{15}}{8}&\frac{47}{8}&0&0&0\\
\frac{1}{8}&\frac{15\sqrt{15}}{8}&4\sqrt{2}&0&0
\end{bmatrix}$
\end{center}
\end{Rem}
\begin{Rem}
In the Theorem \ref{suffi condtn}, if $s_1 = 0$, a sufficient condition for non-negative persymmetric realization of a list of five complex numbers is given as follows:
\begin{enumerate}[(i)]
\item $s_i \geq 0, i=2,3.$
\item $16s_4-5s_2^2 \geq 0$
\item $36s_5-25s_2s_3 \geq 0$
\end{enumerate}
\end{Rem}
\section{Realizability with a perturbation of  Eigenvalues}
In this section we fix all the real eigenvalues and find the bounds on the imaginary part of a complex conjugate pair of eigenvalues required for realizability. Additionally, a perturbation is applied to the set to explore further conditions under which realizability is preserved. This direction is motivated by the work of Guo et al. in \cite{Guo}, who highlight the need for further research on the variation of the imaginary parts of complex eigenvalues. Two cases are considered: a list of four complex numbers with $s_1 = 0$, and a list of five complex numbers consisting of one element with positive real part (the Perron eigenvalue) and all remaining elements with non-positive real parts.
\subsection{\textbf{A Prescribed List of Four Complex Numbers}}
 Consider the set of four complex numbers $\Delta:=\{a,b,c+id,c-id\}$ with $s_1=0,$ where $a,b,c,d \in \mathbb{R}.$ Then $\Delta$ is realizable whenever it satisfies conditions of Theorem \ref{realizing four}. This section analyzes how changes in the imaginary part $d$ of the non-real eigenvalues $c \pm id$ affect the realizability of $\Delta$.
\begin{enumerate}[(i)]
\item \textbf{The real part not equal to zero}: Let $\Delta:=\{a,b,c+id,c-id\}$ with $c \neq 0.$ Note that $s_1$ remains unchanged, while $s_2$, $s_3$, and $4s_4 - s_2^2$ depend on $d.$ 
\begin{align*} 
	s_2&= a^2+b^2+2(c^2-d^2), \,\, s_3=a^3+b^3+2(c^3-3cd^2)\\
	4s_4-s_2^2&= 3(a^4+b^4)-2a^2b^2+4(c^4+d^4)-40c^2d^2-4(a^2+b^2)(c^2-d^2).
\end{align*}
Imposing the conditions of Theorem \ref{realizing four}, the bounds on $d^2$ are given as follows.
\begin{equation}
s_2\geq 0 \implies (a^2+b^2)+2(c^2-d^2) \geq 0 \implies 0\leq d^2\leq c^2 +\frac{(a^2+b^2)}{2}
\end{equation}
which gives an upperbound for $d^2.$
\begin{equation}
s_3\geq 0 \implies (a^3+b^3)+2(c^3-3cd^2) \geq 0 \implies 3cd^2 \leq c^3+\frac{(a^3+b^3)}{2}.
\end{equation}
This can yield both upper and lower bounds for $d^2$, depending on the sign of $c.$ Thus if $c> 0$
\begin{equation}
s_3 \geq 0 \implies  d^2 \leq \frac{c^3+(a^3+b^3)/2}{3c}
\end{equation} 
and if $c< 0$ 
\begin{equation}
s_3 \geq 0\implies d^2 \geq \frac{c^3+(a^3+b^3)/2}{3c}.
\end{equation}
The third condition condition yields 
$4s_4-s_2^2 \geq 0 $ that implies,
\begin{equation}
 4d^4+4(a^2+b^2-10c^2)d^2+3a^4+3b^4-2a^2b^2+4c^4-4(a^2+b^2)c^2 \geq 0.
\end{equation}
This can be viewed as a quadratic equation in $d^2$ in the following way. \[f(d^2)=4d^4+4(a^2+b^2-10c^2)d^2+(3a^4+3b^4-2a^2b^2+4c^4-4(a^2+b^2)c^2).\] The non-negativity of this quadratic equation depends solely on the sign of its discriminant. Considering that the discriminant can be either positive or negative, both cases must be examined.
\begin{enumerate}[(a)]
	\item If the discriminant is negative, $f(d^2) > 0$ for every real $d$.
	\item If the discriminant is positive, the quadratic has two real roots $d_1$ and $d_2.$  Since the leading coefficient is positive, the equation $f(d^2)$ is non-negative outside the interval $(d_1, d_2).$  
	\end{enumerate}
Therefore, we have the range of $d^2$, consequently the range of $d$, for fixed values of $a$, $b$, and $c$.
\item \textbf{The real part equal to zero}: Let $\Delta:=\{a,b,id,-id\}.$ Observe that in this case, both $s_1$ and $s_3$ do not depend on $d$.
Since $s_1 = 0$, it follows that $a = -b$. Thus, 
$s_2 = 2a^2-2d^2$ and $
4s_4-s_2^2= 4d^4+8a^2d^2+4a^4.
$
Also, $4s_4 - s_2^2$ is always non-negative and the bound for $d$ according to Theorem \ref{realizing four} is:
\begin{equation}\label{c not 0}
s_2 \geq 0  \implies 2a^2-2d^2 \geq 0\implies |d|\leq a.
\end{equation}
An upper bound for $d$ is obtained in terms of $a$ in this case.
\end{enumerate}
\begin{exam}
From the previous discussion, the range for $d$ with $a=4,b=-2,c=-1$ is, \[ 0 \leq d \leq \sqrt{11}.\]
\end{exam}
\begin{exam}
For $a=3,b=-3$  and $c=0,$ the range of $d$ is $|d| \leq 3.$ 
\end{exam}
This observation leads to an important conclusion that the imaginary part of the complex eigenvalues cannot be perturbed arbitrarily; there exists a specific bound for the imaginary part within which the set remains realizable. Therefore, variation  on the real eigenvalues must be considered to ensure the realizability of the perturbed set. Subsequently, the conditions required for the preservation of realizability under such changes are examined and the corresponding theorem is presented below.
\begin{theorem}
If $\Delta:=\{a,b,c+id,c-id\}$ is a realizable list where $a$ is the Perron root with $s_1(\Delta)=0$ and $a-b-2d\geq0,$ then the perturbed set  $\Delta^{\prime}:=\{a+t,b-t,c+i(d-t),c-i(d-t)\}$ is also realizable for every $t>0$.
\end{theorem}
\begin{prf}
Assume $s_1(\Delta)=0=s_1(\Delta^\prime)$ so that $c=-\frac{(a+b)}{2}.$ Since $\Delta$ is a realizable list of four complex numbers with $s_1(\Delta)=0$ it satisfies  conditions of Theorem \ref{realizing four} i.e.,
\[s_2(\Delta)\geq 0, \,\, s_3(\Delta)\geq 0, \,\, 4s_4(\Delta)- s_2(\Delta)^2 \geq 0.\]
Note that $\Delta^{\prime}$ is realizable iff
\[s_2(\Delta^{\prime})\geq 0, \,\, s_3(\Delta^{\prime})\geq 0, \,\, 4s_4(\Delta^{\prime})- s_2(\Delta^{\prime})^2 \geq 0.\]
The moments of $\Delta^\prime$ in terms of the moments of $\Delta$ is,
\[s_2(\Delta^{\prime})= s_2(\Delta)+ 2t(a-b+2d),\,\, 
s_3(\Delta^{\prime})= s_3(\Delta)+ 6t^2(a+b)+3t(a+b)(a-b-2d)\]
\begin{align*}
 	4s_4(\Delta^{\prime}) - s_2(\Delta')^2
 	&= (4s_4(\Delta) - s_2(\Delta)^2)
 	+ 16t^4 + 16t^3(a-b-2d) 
 	+ 8t^2\big((a-b-d)^2 + 3d^2\big)+\\
 	& t\Big((a-b)(10a^2+12ab+10b^2+8d^2)
 	+ d \left[12(a^2+ab+b^2)  
 	+ 28ab -16d^2\right]\Big).
 \end{align*}
Assume that $b \geq 0$. From the hypothesis, it is clear that $s_2(\Delta') \geq 0$ and $s_3(\Delta') \geq 0$.  Since  the coefficients of $t^4,$ $t^3$ and $t^2$ are non-negative, we need  the coefficient of $t$ is also non-negative  for the non-negativity of  $4s_4(\Delta') - s_2(\Delta')^2.$ Note that,
$s_2(\Delta)= \frac{3}{2}(a^2+b^2) +ab-2d^2 \geq 0$  this implies $12(a^2+ab+b^2) + 28ab -16d^2 \geq 0$ therefore, $4s_4(\Delta') - s_2(\Delta')^2 \geq 0.$ This shows that $\Delta^{\prime}$ is realizable whenever $b\geq 0.$ If  the perturbation $t$ is bigger than $b$, then $\Delta^{\prime}$ is realizable for $b<0.$  $\hfill{\square}$
\end{prf}
\begin{corollary}
Let $\Delta:=\{a,b,c+id,c-id\}$ be a realizable list where $a$ is the Perron root with $s_1=0$ and $a-b-2d<0.$ Then the perturbed set  $\Delta^{\prime}:=\{a+t,b-t,c+i(d-t),c-i(d-t)\}$ is also realizable for all $t \geq \frac{(b+2d-a)}{4}$.
\end{corollary}

\subsection{\textbf{A prescribed list of five complex numbers}}
In the case of a list of five complex numbers, the spectrum may contain either one conjugate pair of complex eigenvalues or two conjugate pairs. Accordingly, these can be divided into two categories as given below.
\begin{enumerate}[(i)]
\item Consider the set $\Delta:=\{\delta,a_1+ib_1,a_1-ib_1,a_2+ib_2,a_2-ib_2\}$ where $\delta$ is the Perron root with $a_i \leq 0.$ This set consists of two pairs of complex conjugate eigenvalues. If $\Delta$ is realizable, then it satisfies the conditions of Theorem \ref{realising five}. The objective is to determine the range of $(b_1^2 + b_2^2)$ in terms of $\delta$ and $a_i,$ such that the resulting set is realizable. The sum of elements in  $\sigma$ is given by $s_1= \delta+2(a_1+a_2)$ which does not depend on $b_1$ and $b_2.$ Since $\Delta$ is realizable iff $s_1 \geq 0,$ choose  $a_1$ and $a_2$ such that  $\delta+ 2(a_1+a_2) \geq 0.$ The remaining constraints implied by the theorem are
\[s_2= \delta^2+(a_1+ib_1)^2+(a_1-ib_1)^2+(a_2+ib_2)^2+(a_2-ib_2)^2=\delta^2+2(a_1^2+a_2^2-b_1^2-b_2^2) \geq 0\]
i.e.,
\begin{equation}\label{first bound}
	b_1^2+b_2^2 \leq \frac{\delta^2+2(a_1^2+a_2^2)}{2}
\end{equation}
and $5s_2-s_1^2=6(a_1^2+a_2^2)+4\delta^2 -10(b_1^2+b_2^2)-4\delta(a_1+a_2) \geq 0$
i.e.,
\begin{equation}\label{second bound}
	b_1^2+b_2^2 \leq \frac{6(a_1^2+a_2^2)+4\delta^2-4\delta(a_1+a_2)}{10}.
\end{equation}
Using (\ref{first bound}) and (\ref{second bound}) the bounds of $b_1^2+b_2^2$ is given as
\begin{equation}\label{final bound}
0 \leq  b_1^2+b_2^2\leq \min \bigg\{\frac{\delta^2+2(a_1^2+a_2^2)}{2},\frac{6(a_1^2+a_2^2)+4\delta^2-4\delta(a_1+a_2)}{10}\bigg \}.
\end{equation}
If $s_1=0$ we have $\delta=-2(a_1+a_2)$
then, (\ref{final bound}) is 
\[0 \leq b_1^2 + b_2^2 \leq 3a_1^2+ 3a_2^2 + 4a_1a_2.\]
In general, consider the realizable set $\Delta:=\{\delta,a_1+ib_1,a_1-ib_1,a_2+ib_2,a_2-ib_2,\ldots,a_m+ib_m,a_m-ib_m\}$ with cardinality $|\Delta|=2m+1$ and each $a_i \leq 0.$ The objective is to determine the admissible bounds for $\displaystyle \sum_{i=1}^{m}b_i^2$ by fixing $\delta$ and $a_i$ for all $i=1,2,\ldots,m.$ Since $\Delta$ being realizable, the $k^{th}$ moments $s_k$ for $k=1,2$ satisfies
\begin{align*}
s_1&=\delta+2\sum_{i=1}^{m}a_i \geq 0\\
s_2&=\delta^2+2\bigg(\sum_{i=1}^{m}a_i^2-\sum_{i=1}^{m}b_i^2\bigg) \geq 0	\\
(2m+1)s_2 -s_1^2&= (2m+1)\bigg(\delta^2+2\sum_{i=1}^{m}a_i^2-2\sum_{i=1}^{m}b_i^2\bigg)-\bigg(\delta+2\sum_{i=1}^{m}a_i\bigg)^2 \geq 0.
\end{align*}
From these inequalities, the bounds for $\displaystyle\sum_{i=1}^{m}b_i^2$ are obtained as
\[0\leq \sum_{i=1}^{m}b_i^2 \leq \min \bigg\{\frac{m\delta^2-2\delta \left(\sum_{i=1}^{m}a_i\right)-2\left(\sum_{i=1}^{m}a_i\right)^2+(2m+1)\sum_{i=1}^{m}a_i^2}{2m+1} ,\frac{\delta^2 +2\sum_{i=1}^{m}a_i^2}{2}\bigg\}\]
\begin{exam} For the set  $\Delta:=\{4,-1+ib,-1-ib,-1+id,-1-id\}$ with $s_1=0, $ we have $0 \leq b^2+d^2 \leq 10.$
\[(b,d) \in \{(x,y) \in \mathbb{R}^2 : 0 \leq x^2 + y^2 \leq 10 \}.\]
\end{exam}

\item Consider the set $\Delta:=\{\delta,a_1,a_2,a_3+ib_3,a_3-ib_3\}$ of five complex numbers where only one complex conjugate pair occurs. $\delta$ is the Perron root and $a_i \leq 0.$  Let us determine the range of $b_3$ for which $\Delta$ is realizable. Assume $\Delta$ is realizable, then it satisfies conditions of Theorem \ref{realising five}, given as
$s_1 \geq 0,\,\, s_2 \geq 0,\,\, 5s_2-s_1^2\geq 0.$
The inequality \[s_1 \geq 0 \implies \delta+a_1+a_2+2a_3 \geq 0\]
which is independent of $b_3.$
\begin{align*}
	s_2 \geq 0 &\implies \delta^2+a_1^2+a_2^2+(a_3+ib_3)^2+(a_3-ib_3)^2 \geq 0\\
	&\implies b_3^2 \leq \frac{\delta^2+a_1^2+a_2^2+2a_3^2}{2}. 
\end{align*}
Non-negativity of $5s_2-s_1^2$ implies,
$5\delta^2+a_1^2+a_2^2+2(a_3^2-b_3^2) - (\delta+a_1+a_2+2a_3)^2 \geq 0. $ Therefore,
\[b_3^2 \leq \frac{4(\delta^2+a_1^2+a_2^2)+6a_3^2-2(\delta a_1+\delta a_2+a_1a_2)-4a_3(\delta+a_1+a_2)}{10}.\]
Hence, the upper bounds of $b_3$ are $M_1:=\sqrt{\frac{4(\delta^2+a_1^2+a_2^2)+6a_3^2-2(\delta a_1+\delta a_2+a_1a_2)-4a_3(\delta+a_1+a_2)}{10}}$ and  $M_2:=\sqrt{\frac{\delta^2+a_1^2+a_2^2+2a_3^2}{2}}$
The admissible range of $b_3$ is therefore, $|b_3| \leq \min\{M_1,M_2\}.$
\begin{exam}
Consider $\Delta:=\{8,-3,-1,-1+ib_3,-1-ib_3\}$
with $s_1=2.$ The upper bounds are,
$M_1=\sqrt{38}$  and $M_2=\sqrt{\frac{188}{5}}.$
The bound for $d$ is 	$|b_3| \leq \sqrt{\frac{188}{5}}$
\end{exam}
\end{enumerate}
From the above analysis, it is clear that perturbation of the imaginary parts up to a certain bound, while keeping the other eigenvalues fixed, results in a realizable spectrum. Another aspect to consider is whether realizability is preserved when the spectrum is perturbed  in opposite directions and simultaneously decrease the imaginary part of a complex conjugate pair. This is considered in the following theorem. 
\begin{theorem}\label{perturb thm}
Let $\Delta:=\{\delta,a,b,c+id,c-id\}$ be a realizable list where $\delta$ is the Perron root with $a$ and $b$ non-positive. Then the perturbed set $\Delta^{\prime}:=\{\delta+t,a-t,b,c+i(d-t),c-i(d-t)\}$ is also realizable for every $t>0.$
\end{theorem}
\begin{prf}
Since $\Delta$ is realizable, it satisfies the conditions of Theorem \ref{realising five}. It follows that,
$s_1(\Delta)\geq 0, \, s_2(\Delta) \geq 0,\, 5s_2(\Delta)-s_1(\Delta)^2 \geq 0.$  Note that, $s_1(\Delta)=s_1(\Delta^{\prime}).$ Therefore $\Delta^{\prime}$ is realizable iff
$s_2(\Delta^{\prime})\geq 0, \, 5s_2(\Delta^{\prime})-s_1(\Delta^{\prime})^2 \geq 0.$
Write these inequalities in terms of the moments of $\Delta$ yields,
\[ s_2(\Delta^{\prime})=s_2(\Delta)+2t(\delta-a+2d)\]
\[5s_2(\Delta^{\prime})-s_1(\Delta^{\prime})^2=5s_2(\Delta)-s_1(\Delta)^2+10(\delta-a+2d).\]
Since these are non-negative by hypothesis, $\Delta^{\prime}$ is realizable for all $t \geq 0.$  $\hspace{3cm}\hfil{\square}$
\end{prf}
Suppose there are two conjugate pairs of complex eigenvalues, perturbations can be applied to both the real and imaginary parts, as well as to the Perron root. The following theorem presents the corresponding perturbation result.
\begin{theorem}
Let $\Delta:=\{\delta,a+ib,a-ib,c+id,c-id\}$ be a realizable list where $\delta$ is the Perron root with $a$ and $b$ are non-positive. Then the perturbed set $\Delta^{\prime}:=\{\delta+2t,(a-t)+i(b-t),(a-t)-i(b-t),c+id,c-id\}$ is also realizable for every $t>0.$
\end{theorem}
\begin{prf}
This follows by the same argument as the proof of Theorem \ref{perturb thm}. The expressions for $s_2$ and $5s_2 - s_1^2$ of $\Delta^{\prime}$ are as follows:
\begin{align*}
s_2(\Delta^{\prime})&=s_2(\Delta)+4(t^2+t(\delta-a+b))\\	
5s_2(\Delta^{\prime})-s_1(\Delta^{\prime})^2&=(5s_2(\Delta)-s_1(\Delta)^2)+20(t^2+t(\delta-a+b))
\end{align*}
where both are non-negative by hypothesis. Therefore, $\Delta^{\prime}$ is realizable for all $t > 0 .\hfil{\square}$
\end{prf}
In general, the above result extends to $n$ values consisting of one element with positive real part (the Perron eigenvalue) and all remaining elements with non-positive real parts, as follows:
\begin{theorem}
Let $\Delta:=\{\delta,a_1+ib_1,a_1-ib_1,a_2+ib_2,a_2-ib_2,\ldots,a_m+ib_m,a_m-ib_m\}$ where  $\delta$ is the Perron root and $a_i \leq 0$ for $i=1,2,\ldots,m.$  Then the perturbed set $\Delta^{\prime}:=\{\delta+2t,a_1+ib_1,a_1-ib_1,a_2+ib_2,a_2-ib_2,\dots,(a_k-t)+i(b_k-t),(a_k-t)-i(b_k-t),\ldots,a_m+ib_m,a_m-ib_m\}$ is also realizable for any $k$ and for all $t \geq 0.$
\begin{prf}
The proof follows the same approach as that of Theorem \ref{perturb thm}. The corresponding expressions for $s_2$ and $5s_2 - s_1^2$ of $\Delta^{\prime}$ are given by:
\begin{align*}
s_2(\Delta^{\prime})&=s_2(\Delta)+4(t^2+t(\delta-a_k+b_k))\\
(2m+1)s_2(\Delta^{\prime})-s_1(\Delta^{\prime})^2&=((2m+1)s_2(\Delta)-s_1(\Delta)^2)+4(2m+1)(t^2+t(\delta-a_k+b_k)).
\end{align*}
Since both are non-negative by hypothesis, $\Delta^{\prime}$ is realizable for all $t > 0 .\hfil{\square}$
\end{prf}
\end{theorem}
\section{Conclusion}
To summarize, the results presented here primarily highlight the equivalence of the Non-negative Inverse Eigenvalue Problem (NIEP) and the Persymmetric Non-negative Inverse Eigenvalue Problem (PNIEP) for lists of five complex numbers with zero sum. Additionally, a sufficient condition for the existence of a persymmetric non-negative matrix corresponding to a prescribed list of complex numbers, provided the list satisfies certain conditions, has been established. Furthermore, another sufficient condition has been proven for the existence of a non-negative persymmetric matrix realizing a list of five complex numbers. A detailed investigation has also been conducted on the range of imaginary part of the complex eigenvalues, explores how small perturbations in this component influence the realizability of the spectrum. Perturbation results have also been established for cases where the imaginary part of complex eigenvalues is varied.
\bibliographystyle{elsarticle-num}  

\end{document}